\documentclass[12pt]{amsart}

\usepackage[margin=3cm]{geometry}

\usepackage{amssymb}

\usepackage{enumerate}

\usepackage{graphicx}

\makeatletter
\@namedef{subjclassname@2010}{%
  \textup{2010} Mathematics Subject Classification}
\makeatother

\newtheorem{thm}{Theorem}[section] 

\newtheorem{lem}{Lemma}[section]
\newtheorem{prop}{Proposition}[section]

\theoremstyle{definition}

\numberwithin{equation}{section}

\begin{document}

\title[Improving an inequality for the divisor function]{Improving an inequality for the divisor function}

\author[J. Lay]{Jeffrey P.S. Lay}
\address{Mathematical Sciences Institute, The Australian National University}
\email{jeffrey.lay@anu.edu.au}

\date{}

\begin{abstract}
We improve using elementary means an explicit bound on the divisor function due to Friedlander and Iwaniec. Consequently we modestly improve a result regarding a sieving inequality for Gaussian sequences.
\end{abstract}

\subjclass[2010]{11N36, 11N37, 11N56, 11N64}

\keywords{divisor function, small divisors, Gaussian sequences, sieve estimates}

\maketitle


\section{Introduction}

Let $\tau(n)$ be the number of divisors of $n$. While asymptotic estimates for weighted sums $\sum \tau(n)a_n$ are generally difficult to obtain, explicit bounds often suffice in applications. 

We shall consider the relationship between $\tau(n)$ and averages of $\tau(d)$ for small divisors $d$ of $n$. Landreau~\cite{Landreau1988} showed that for any integer $k\geq 2$ there exists a constant $C_k >0$ such that
\begin{equation} \label{e81}
	\tau(n) \leq C_k \sum_{\substack{d|n \\ d\leq n^{1/k}}} \big( 2^{\omega(d)} \tau(d) \big)^k, \qquad n\geq 1,
\end{equation}	
where $\omega(n)$ counts the number of distinct primes dividing $n$. We shall like to make the constants $C_k$ effective. Friedlander and Iwaniec~\cite{OperadeCribro} considered, {\em inter alia}, a weakened version of~\eqref{e81} for $k=4$, making use of the trivial bound $2^{\omega(n)} \leq \tau(n)$. They showed
\begin{equation} \label{e82}
	\tau(n) \leq C \sum_{\substack{d|n \\ d\leq n^{1/4}}} \tau(d)^8, \qquad n\geq 1,
\end{equation}
holds for $C=256$. Numerical evidence suggests this constant is far from optimal. In fact, it can be easily verified that~\eqref{e82} holds with $C=8$ for all $1 \leq n \leq 10^8$. Moreover, equality is attained for $733,133$ values of $n$ within this interval, these being the square-free numbers $n=p_1p_2p_3$ satisfying $n^{1/4}<\min(p_1,p_2,p_3)$. So for small $n$ it is certainly the case that $C=8$ is the best possible constant, with evidence suggesting this trend should continue as $n\to \infty$. Our aim is to investigate whether $C\leq 8$ is admissible for all $n$ sufficiently large, as well as whether the sum can be made sharper.

In this article we show that~\eqref{e82} indeed holds for $C=8$. In addition we improve on the exponent of $\tau(d)$ in the sum, which~\eqref{e81} suggests should be much smaller than $8$, at least for non square-free $n$. Our main result to reach this goal is the following.

\begin{thm} \label{t1}
Let $n \geq 1$. Then there exists $d\leq n^{1/4}$ with $d|n$ such that $\tau(n) \leq 8 \tau(d)^7$.
\end{thm}

We shall also show that the constant $C$ in~\eqref{e82} must satisfy $C\geq 8$. Thus we deduce

\begin{thm} \label{t2}
We have 
\begin{equation*}
	\tau(n) \leq 8 \sum_{\substack{d|n \\ d\leq n^{1/4}}} \tau(d)^7, \qquad n\geq 1,
\end{equation*}
the constant $8$ being best possible for all $n$.
\end{thm}

The consideration of~\eqref{e82} by Friedlander and Iwaniec in~\cite{OperadeCribro} led to the study of sieving inequalities for Gaussian sequences. We shall see in \S{6} how Theorem~\ref{t2} may be used to modestly improve one of their results~\cite{FriedlanderIwaniec2007}.

\subsection*{Acknowledgements}

The author would like to thank Tim Trudgian for suggesting the problem and for providing awesome feedback. The author is supported by an Australian Government Research Training Program (RTP) Scholarship.

\section{A lower bound}

Our first result describes a natural lower bound for the constant $C$ in~\eqref{e82}. This bound arises from the consideration of a particular set of square-free numbers. In fact, the result extends to the general case~\eqref{e81}.

\begin{prop} \label{p21}
Fix an integer $k\geq 2$. Then for any multiplicative function $f:\mathbb N \to \mathbb R$ we have
\begin{equation*}
	\limsup_{n\to \infty} \, \tau(n) \Bigg( \sum_{\substack{d|n \\ d\leq n^{1/k}}} f(d) \Bigg)^{-1} \geq 2^{k-1}.
\end{equation*}
\end{prop}

\begin{proof}
Take a prime $p_1> 2^{(k-1)(k-2)/2}$ and choose, using Bertrand's postulate, primes $p_2<p_3<\dots <p_{k-1}$ such that $p_1<p_2<2 p_1$ and $p_i<2^{i-1} p_1$ for $3\leq i \leq k-1$. Then
\begin{equation*}
	p_1^{k-1}>2^{(k-1)(k-2) \over 2} \times p_1^{k-2}=\prod_{i=2}^{k-1} 2^{i-1} p_1 >p_2p_3\dotsm p_{k-1}.
\end{equation*}
Consider now $n=p_1p_2\dotsm p_{k-1}$. We see that $p_1 > n^{1/k}$, whence there are no non-trivial divisors $d$ of $n$ with $d\leq n^{1/k}$. So for such an $n$ we have $\tau(n)=2^{k-1}$ and
\begin{equation*}
	\sum_{\substack{d|n \\ d\leq n^{1/k}}} f(d) = f(1) = 1.
\end{equation*}
\end{proof}



\section{Some upper bounds}

We now turn our attention to proving Theorem~\ref{t1}. The aim is to choose for any $n$ a divisor $d\leq n^{1/4}$ for which $\tau(d)$ is as large as possible. In this section we demonstrate this procedure for $n$ with certain prime factorisations.

We shall make use of the following elementary inequalities. We write $[x]$ for the integer part of $x$.

\begin{lem} \label{l9}
For all integers $t \geq 4$ we have $7 [t/4] \geq t$ and $([t/4]+1)^4 \geq 2(t+1)$.
\end{lem}

\begin{proof}
Let $i\geq 1$ be the unique integer such that $4i \leq t \leq 4i+3$. For the first inequality we simply see that $7[t/4]=7i \geq 4i+3 \geq t$. For the second we have $([t/4]+1)^4 = (i+1)^4 \geq 8(i+1) = 2(4i+3)+2 \geq 2t+2$.
\end{proof}

We consider the various cases pertaining to how prime powers appear in the prime factorisation of $n$. Our first lemma deals with the case when all exponents are at least~$4$.

\begin{lem} \label{l70}
Suppose $n=p_1^{a_1}p_2^{a_2} \dotsm p_t^{a_t}$ with $a_i \geq 4$ for all $1\leq i \leq t$. Then there exists $d\leq n^{1/4}$ with $d|n$ such that $\tau(n)\leq 2^{-t} \tau(d)^4$. 
\end{lem}

\begin{proof}
We let $d = \prod_{i=1}^t p_i^{[a_i/4]}$. Then $d \leq n^{1/4}$ and we have by Lemma~\ref{l9}
\begin{equation*}
	\tau(d)^4 = \prod_{i=1}^t \bigg( \bigg[{a_i \over 4} \bigg] + 1 \bigg)^4 \geq 2^t \prod_{i=1}^t (a_i+1) = 2^t \tau(n).
\end{equation*}
\end{proof}

We now consider the cases when all prime powers appearing in the prime factorisation of $n$ occur with exponent $k$ for $k\in \{1,2,3\}$.

\begin{lem} \label{l71}
	Suppose $n=p_1p_2 \dotsm p_t$ with $p_1<p_2<\dots <p_t$. Then there exists $d \leq n^{1/4}$ with $d|n$ such that
		\begin{equation*}
			\tau(n) \leq 
				\begin{cases}
					2^t \tau(d) & \text{if } t \in 								\{1,2,3\}, \\
					\tau(d)^7 & \text{if } t \geq 	4.
				\end{cases}
		\end{equation*}
\end{lem}

\begin{proof}
Firstly let $t \in \{1,2,3\}$ be fixed. In each of these cases we let $d=1$. Then $2^t \tau(d) = \tau(n)$. 

On the other hand, if $t\geq 4$ we take $d=p_1p_2\dots p_{[t/4]}$. Then $d\leq n^{1/4}$ and we have, by Lemma~\ref{l9}, $\tau(d)^7 = 2^{7\times {[t/4]}} \geq 2^t = \tau(n)$.
\end{proof}

\begin{lem} \label{l72}
Suppose $n=p_1^2 p_2^2 \dots p_t^2$ with $p_1<p_2<\dots <p_t$. Then there exists $d\leq n^{1/4}$ with $d|n$ such that
\begin{equation*}
	\tau(n) \leq 
		\begin{cases}
			3\tau(d) & \text{if } t=1, \\
			2^{-2} \tau(d)^7 & \text{if } t \in \{2,3\}, \\
			\tau(d)^7 & \text{if } t \geq 4.
		\end{cases}
\end{equation*}
\end{lem}

\begin{proof}
If $t=1$ we let $d=1$. Then $3 \tau(d) = \tau(p_1^2)=\tau(n)$. 

Next suppose $t \in \{2,3\}$. In these cases take $d=p_1$, whence $\tau(d)^7 = 2^7 > 2^2 \times 3^3 \geq 2^2 \tau(n)$.

Finally suppose $t\geq 4$. Take $d=p_1^2p_2^2 \dotsm p_{[t/4]}^2$. Then $\tau(d)^7 = 3^{7 \times [t/4]} \geq 3^t = \tau(n)$.
\end{proof}

\begin{lem} \label{l73}
Suppose $n=p_1^3 p_2^3 \dotsm p_t^3$ with $p_1<p_2<\dots<p_t$. Then there exists $d\leq n^{1/4}$ with $d|n$ such that
\begin{equation*}
	\tau(n) \leq
		\begin{cases}
			4 \tau(d) & \text{if } t=1, \\
			2^{-3} \tau(d)^7 & \text{if } t=2, \\
			2^{-5} \tau(d)^7 & \text{if } t=3, \\
			\tau(d)^7 & \text{if } t\geq 4.
		\end{cases}
\end{equation*}
\end{lem}

\begin{proof}
As before, if $t=1$ let $d=1$, whence $4 \tau(d) = \tau(n)$. 

If $t=2$ we take $d=p_1$, which gives $\tau(d)^7 = 2^7 = 2^3 \tau(n)$. 

Next, if $t=3$ let $d=p_1^2$. Then $\tau(d)^7 = 3^7 > 2^5 \times 4^3 = 2^5 \tau(n)$.

Finally suppose $t\geq 4$. Take $d=p_1^3p_2^3 \dotsm p_{[t/4]}^3$, whence $\tau(d)^7 = 4^{7\times [t/4]} \geq 4^t = \tau(n)$.
\end{proof}

We are now ready to combine these estimates to prove Theorem~\ref{t1}.

\section{Proof of Theorem~\ref{t1}}

Let $n\geq 1$ and consider the unique prime factorisation of $n$. We group the prime powers according to their exponents: for each $i \in \{1,2,3\}$ let $m_i$ be the product of those occurring with exponent $i$ and let $l$ be the product of those with exponent at least $4$. Henceforth the relations $m_i=1$ and $l=1$ will be understood to mean that no primes of the corresponding form divide $n$. 

Write $n=m_1m_2m_3l$. First observe by Lemma~\ref{l70} that there exists a divisor $d_l$ of $l$ with $d_l\leq l^{1/4}$ for which
\begin{equation} \label{e678}
	\tau(n) = \tau(m_1m_2m_3) \tau(l) \leq \tau(m_1m_2m_3) \tau(d_l)^7.
\end{equation}
Thus to prove our theorem it suffices to consider those $n$ whose prime factorisations consist solely of prime powers with exponents strictly less than $4$. That is, if for each such $n=m_1m_2m_3$ we can find a divisor $d\leq n^{1/4}$ with $\tau(n) \leq 8 \tau(d)^7$ then by~\eqref{e678} we are done.

In each of the following cases the numbers $d_1,d_2,d_3$ are chosen as per Lemmas~\ref{l71},~\ref{l72}, and ~\ref{l73}. Note that these satisfy $d_i | m_i$ and $d_i\leq m_i^{1/4}$. Moreover, if $m_i=1$ we may choose $d_i=1$.

\begin{enumerate}[(I)]
\item Let $m_1\geq 1$.
	\begin{enumerate}[(i)]
		\item If $\omega(m_2) \in \{2,3\}$ then $\tau(n) \leq 8 \tau(d_1)^7 \times 2^{-2} \tau(d_2)^7 \times 4 \tau(d_3)^7 \leq 8 \tau(d_1d_2d_3)^7$.
		\item If $\omega(m_3) \in \{2,3\}$ then $\tau(n) \leq 8 \tau(d_1)^7 \times 3\tau(d_2)^7 \times 2^{-3} \tau(d_3)^7 \leq 3 \tau(d_1d_2d_3)^7$.
	\end{enumerate}
Henceforth we only consider the cases $m_2,m_3=1$ and $\omega(m_2),\omega(m_3) \in \mathbb N \setminus \{2,3\}$.

\item Suppose $m_1=1$ or $\omega(m_1)\geq 4$. 
	\begin{enumerate}[(i)]
		\item If at least one of $\omega(m_2)\geq 4$ or $\omega(m_3)\geq 4$ holds then $\tau(n) \leq \tau(d_1)^7 \times 4 \tau(d_2)^7 \times \tau(d_3)^7 = 4 \tau(d_1d_2d_3)^7$.
		\item On the other hand, suppose $\omega(m_2)=\omega(m_3)=1$. In this case write $n=m_1p_1^2p_2^3$. Let $d'=\min(p_1,p_2) \leq (p_1^2p_2^3)^{1/4}$. Then $\tau(d')^7=2^7 > \tau(p_1^2p_2^3)$ and so $\tau(n) < \tau(d_1)^7 \times \tau(d')^7 \leq \tau(d_1d')^7$.
	\end{enumerate}

\item Suppose $\omega(m_1)=1$. 
	\begin{enumerate}[(i)]
		\item If at least one of $\omega(m_2)\geq 4$ or $\omega(m_3)\geq 4$ holds then $\tau(n) \leq 2\tau(d_1)^7 \times 4\tau(d_2)^7 \times \tau(d_3)^7 \leq 8\tau(d_1d_2d_3)^7$.
		\item On the other hand, suppose $\omega(m_2)=\omega(m_3)=1$. Write $n=m_1 p_1^2p_2^3$. Let $d'=\min(p_1,p_2) \leq (p_1^2p_2^3)^{1/4}$. Then $\tau(d')^7 > \tau(p_1^2p_2^3)$ and so $\tau(n) < 2\tau(d_1) \times \tau(d')^7 \leq 2 \tau(d_1d')^7$.
	\end{enumerate}

\item Suppose $\omega(m_1)=2$. 
	\begin{enumerate}[(i)]
		\item If $\omega(m_2)\geq 4$ and $\omega(m_3)\geq 4$ then $\tau(n) \leq 4 \tau(d_1) \times \tau(d_2)^7 \times \tau(d_3)^7 \leq 4 \tau(d_1d_2d_3)^7$.
		\item If $\omega(m_2)=1$ and $\omega(m_3)\geq 4$ write $n=p_1p_2p_3^2 m_3$. Let $d'=\min(p_1,p_2,p_3)\leq (p_1p_2p_3^2)^{1/4}$. Then $\tau(d')^7 > \tau(p_1p_2p_3^2)$ and so $\tau(n) < \tau(d')^7 \times \tau(d_3)^7 = \tau(d'd_3)^7$.
		\item If $\omega(m_2)\geq 4$ and $\omega(m_3)=1$ write $n=p_1p_2 p_3^3 m_2$. Let $d'=\min(p_1,p_2,p_3)\leq (p_1p_2p_3^3)^{1/4}$. Then $\tau(d')^7 > \tau(p_1p_2p_3^3)$ and so $\tau(n) < \tau(d')^7 \times \tau(d_2)^7 = \tau(d'd_2)^7$.
		\item Suppose $\omega(m_2)=\omega(m_3)=1$. Write $n=m_1p_1^2p_2^3$. Let $d'=\min(p_1,p_2)\leq (p_1^2p_2^3)^{1/4}$. Then $\tau(d')^7 > \tau(p_1^2p_2^3)$ and so $\tau(n) < 4\tau(d_1) \times \tau(d')^7 \leq 4\tau(d_1d')^7$.
	\end{enumerate}
	
\item Suppose $\omega(m_1)=3$. 
	\begin{enumerate}[(i)]
		\item If $\omega(m_2)\geq 4$ and $\omega(m_3)\geq 4$ then $\tau(n) \leq 8\tau(d_1) \times \tau(d_2)^7 \times \tau(d_3)^7 \leq 8\tau(d_1d_2d_3)^7$.
		\item If $\omega(m_2)=1$ and $\omega(m_3)\geq 4$ write $n=p_1p_2p_3p_4^2m_3$. Let $d'=\min(\{p_i\})\leq (p_1p_2p_3p_4^2)^{1/4}$. Then $\tau(d')^7 > \tau(p_1p_2p_3p_4^2)$ and so $\tau(n) < \tau(d')^7 \times \tau(d_3)^7 = \tau(d'd_3)^7$.
		\item If $\omega(m_2)\geq 4$ and $\omega(m_3)=1$ write $n=p_1p_2p_3p_4^3m_2$. Let $d' = \min(\{p_i\}) \leq (p_1p_2p_3p_4^3)^{1/4}$. Then $\tau(d')^7 > \tau(p_1p_2p_3p_4^3)$ and so $\tau(n) < \tau(d')^7 \times \tau(d_2)^7 = \tau(d'd_2)^7$.
		\item If $\omega(m_2)=\omega(m_3)=1$ write $n=m_1 p_1^2p_2^3$. Let $d'=\min(p_1,p_2)\leq (p_1^2p_2^3)^{1/4}$. Then $\tau(d')^7 > \tau(p_1^2p_2^3)$ and so $\tau(n) < 8\tau(d_1) \times \tau(d')^7 \leq 8 \tau(d_1d')^7$.
	\end{enumerate}
\end{enumerate}

\section{Further speculation}

Returning to~\eqref{e81} one may consider for any $k\geq 2$ and $\eta \geq 1$ the generalised inequality
\begin{equation} \label{e91}
	\tau(n) \leq C_{k,\eta} \sum_{\substack{d|n \\ d\leq n^{1/k}}} \tau(d)^\eta.
\end{equation}
Clearly if~\eqref{e91} holds then it must also be true for any $\eta'>\eta$, in which case we may choose $C_{k,\eta'}=C_{k,\eta}$. Thus for fixed $k$ and $C_k=C_{k,\eta}$ we would like to know the smallest $\eta$ for which~\eqref{e91} holds. 

A natural question to consider in our case is whether Theorem~\ref{t1} can be improved to show that for all $n\geq 1$ there exists a divisor $d \leq n^{1/4}$ such that $\tau(n) \leq 8 \tau(d)^6$. It appears, however, that the purely elementary methods presented in this paper cannot achieve this in any practical sense. To see why consider a number $n=p_1^2 p_2^2 \dotsm p_{t_1}^2 q_1^3 q_2^3 \dotsm q_{t_2}^3$ with $t_1 \geq 4$ and $t_2\geq 4$. Suppose $p_1<p_2<\dots <p_{t_1}$ and $q_1<q_2<\dots <q_{t_2}$. Without additional assumptions on $n$ the best choice of divisor $d\leq n^{1/4}$ for which $\tau(d)$ is as large as possible is $d=p_1^2 p_2^2 \dotsm p_{[t_1/4]}^2 q_1^3 q_2^3 \dotsm q_{[t_2/4]}^3$. But then (cf.\ Lemma~\ref{l9})
\begin{equation*}
	\tau(d)^6 = 3^{6 \times \left[{t_1 \over 4}\right]} \times 4^{6 \times \left[{t_2 \over 4} \right]} \geq 3^{t_1-1} \times 4^{t_2-1} = 12^{-1} \tau(n).
\end{equation*}
Thus the best estimate we can produce unconditionally is $\tau(n) \leq 12 \tau(d)^6$. One may enumerate each of the various cases in regards to the relative sizes of the $p_i,q_j$ to produce a divisor $d$ with $\tau(d)$ large enough; this seems a formidable task in general.

In any case it remains an open problem to determine the smallest $\eta>0$ such that
\begin{equation} \label{e909}
	\tau(n) \ll_\eta \sum_{\substack{d|n \\ d \leq n^{1/4}}} \tau(d)^\eta.
\end{equation}
At least in the square-free case this problem has been solved. Iwaniec and Munshi~\cite{IwaniecMunshi2010} showed that~\eqref{e909} holds for square-free $n$ with any $\eta>3 \log 3/\log 2 - 4=0.75488\dots$, this lower bound being best possible.

\section{An application to Gaussian sequences}

Of significant interest in sieve theory is the detection of primes in Gaussian sequences, viz.\ sequences supported on integers which can be expressed as the sum of two squares. 

Here we consider a generalised Gaussian sequence $\mathcal A=(a_n)$ defined by
\begin{equation} \label{e321}
	a_n = \sum_{\substack{l^2+m^2=n \\ (l,m)=1}} \gamma_l,
\end{equation}
where $l,m$ run over positive integers and $\gamma_l$ are any complex numbers with $|\gamma_l| \leq 1$. We further suppose that the $\gamma_l$ are supported on $r$-th powers, i.e., $\gamma_l=0$ if $l \not= k^r$.

In the process of sieving $\mathcal A$ one requires good estimates for
\begin{equation} \label{e322}
	A_d(x) = \sum_{\substack{n\leq x \\ d|n}} a_n.
\end{equation}
It can be shown (see equations~(6) and~(7) in~\cite{FriedlanderIwaniec2007}) that 
\begin{equation*}
	\sum_{n\leq x} a_n = \sum_{l < \sqrt x} \gamma_l {\varphi(l) \over l} \sqrt{x-l^2} + O(x^{1 \over 2r} \log x ),
\end{equation*}
so for $d$ not too large we expect $A_d(x)$ to be uniformly well approximated by
\begin{equation*}
	M_d(x) = {\rho(d) \over d} \sum_{\substack{l<\sqrt x \\ (l,d)=1}} \gamma_l {\varphi(l) \over l} \sqrt{x-l^2},
\end{equation*}
where $\rho(d)$ is the number of solutions to the congruence $\nu^2+1 \equiv 0 \mod d$. 

To estimate~\eqref{e322} we may consider instead the smoothed sum
\begin{equation*}
	A_d(f) = \sum_{n \equiv 0 \bmod d} a_n f(n),
\end{equation*}
where $f \in C^\infty([0,\infty))$ is such that $f(t)=1$ if $0\leq t \leq (1-\kappa)x$ and $f(t)=0$ if $t\geq x$. Here $x^{-1/4r} \leq \kappa \leq 1$ is some parameter to be optimised later. 

\begin{prop} \label{p99}
Suppose $\sqrt x \leq D \leq x^{(r+1)/(2r)}$. Then
\begin{equation*}
	\sum_{d\leq D} |A_d(x) - A_d(f)| \ll \kappa x^{r+1 \over 2r} (\log x)^{128}.
\end{equation*}
\end{prop}

\begin{proof}
A rearrangement of the sum gives
\begin{equation*}
\begin{aligned}
	\sum_{d\leq D} |A_d(x)-A_d(f)| &=
	\sum_{d\leq D} \Bigg| \sum_{\substack{(1-\kappa)x<n\leq x \\ d|n}}  \big( 1-f(n) \big) \sum_{\substack{l^2+m^2=n \\ (l,m)=1}} \gamma_l \bigg| \\
	&\ll \sum_{\substack{(1-\kappa)x<l^2+m^2\leq x \\ (l,m)=1}} |\gamma_l| \sum_{d|(l^2+m^2)} 1 \\
	&\ll \sideset{}{'}\sum_{\substack{(1-\kappa)x<l^2+m^2\leq x \\ (l,m)=1}} |\gamma_l| \tau(l^2+m^2) + \sqrt x \log x,
\end{aligned}
\end{equation*}
where $\sum{}^{'}$ means that the terms with a value of $l$ which is nearest to $\sqrt x$ are omitted. We deduce from Theorem~\ref{t2} that 
\begin{equation*}
	\sideset{}{'}\sum_{\substack{(1-\kappa)x<l^2+m^2\leq x \\ (l,m)=1}} |\gamma_l| \tau(l^2+m^2) \ll \sideset{}{'}\sum_{l<\sqrt x} |\gamma_l| \sum_{\substack{d \leq x^{1/4} \\ (d,l)=1}} \tau(d)^7 \sum_{\substack{(1-\kappa)x < l^2+m^2 \leq x \\ l^2+m^2 \equiv 0 \bmod d}} 1.
\end{equation*}
Now split the range of $m$ into residue classes $m \equiv \nu l \mod d$, where $\nu^2+1 \equiv 0 \mod d$. This, combined with the observation that $m$ runs over an interval of length $O(\kappa x/\sqrt{x-l^2})$, allows us to estimate the above by
\begin{equation*}
\begin{aligned}
	&\ll \kappa x \Bigg( \sum_{d\leq x^{1/4}} \tau(d)^7 {\rho(d) \over d} \Bigg) \Bigg( \sideset{}{'}\sum_{l<\sqrt x} {|\gamma_l| \over \sqrt{x-l^2}} \Bigg) + x^{{1 \over 4}+{1 \over 2r}} (\log x)^{128} \\
	&\ll \kappa x \times (\log x)^{128} \times x^{1-r \over 2r} + x^{{1 \over 4}+{1 \over 2r}} (\log x)^{128} \\
	&\ll \kappa x^{r+1 \over 2r} (\log x)^{128}.
\end{aligned}
\end{equation*}
\end{proof}

We now use Proposition~\ref{p99} to improve the error term in the main theorem of~\cite{FriedlanderIwaniec2007} by a factor of $O\big( (\log x)^{64.75} \big)$. 

\begin{thm}
Let $a_n$ and $A_d(x)$ be as in~\eqref{e321} and~\eqref{e322}, respectively. Suppose $\sqrt x \leq D \leq x^{(r+1)/(2r)}$. Then
\begin{equation*}
	\sum_{d\leq D} | A_d(x) - M_d(x) | \ll D^{1 \over 4} x^{3(r+1) \over 8r} (\log x)^{65.25}.
\end{equation*}
\end{thm}

\begin{proof}
We combine equations~(19) and (35) from~\cite{FriedlanderIwaniec2007} with Proposition~\ref{p99} above to obtain the estimate
\begin{equation*}
\begin{aligned}
	\sum_{d\leq D} | A_d(x) - M_d(x) | &\ll \sum_{d\leq D} |A_d(x)- A_d(f)| + \kappa^{-1} D^{1 \over 2} x^{r+1 \over 4r} (\log x)^{5 \over 2} + \kappa x^{r+1 \over 2r} \log x \\
	&\ll \kappa x^{r+1 \over 2r} (\log x)^{128} + \kappa^{-1} D^{1 \over 2} x^{r+1 \over 4r} (\log x)^{5 \over 2}.
\end{aligned}
\end{equation*}
Choosing 
\begin{equation*}
	\kappa = D^{1 \over 4} x^{-{r+1 \over 8r}} (\log x)^{{5 \over 4}-64}
\end{equation*}
yields the desired result.
\end{proof}



\begin{thebibliography}{99}


\bibitem{FriedlanderIwaniec2007}
	J. Friedlander and H. Iwaniec,
	\emph{Gaussian sequences in arithmetic progressions},
	Funct. Approx. Comment. Math.
	{37}
	(2007),
	no.~1,
	149--157

\bibitem{OperadeCribro}
	J. Friedlander and H. Iwaniec,
	\emph{Opera de Cribro},
	American Mathematical Society, Providence,
	~(2010)
	
\bibitem{IwaniecMunshi2010}
	H. Iwaniec and R. Munshi,
	\emph{Cubic polynomials and quadratic forms},
	J. Lond. Math. Soc. (2)
	{81}
	(2010)
	no.~1,
	45--64
	
\bibitem{Landreau1988}
	B. Landreau,
	\emph{Majorations de fonctions arithm\'etiques en moyenne sur des ensembles de faible densit\'e},
	In S\'eminaire de Th\'eorie des Nombres, 1987--1988



\end{thebibliography}

\end{document}